\documentclass[leqno,12pt]{amsart}
\usepackage{amsmath}
\usepackage{amssymb,amscd}
\usepackage{amsthm}
\usepackage{latexsym}
\usepackage[myheadings]{fullpage}

\usepackage{color}  

\newtheorem{thm}{Theorem}[section]
\newtheorem{lem}[thm]{Lemma}
\newtheorem{prop}[thm]{Proposition}
\newtheorem{cor}[thm]{Corollary}
\newtheorem{rem}[thm]{Remark}

\newtheorem{ex}[thm]{Example}

\newcommand{\R}{\mathcal R_+}
\newcommand{\m}{\mathfrak m}
\newcommand{\p}{\mathfrak p}

\newcommand{\Hom}{{\rm Hom}}
\newcommand{\Ker}{{\rm Ker}}
\newcommand{\RR}{R(I)_{a,b}}
\newcommand{\om}{\omega_R}

\def\opn#1#2{\def#1{\operatorname{#2}}}
\opn\depth{depth}
\opn\height{height}
\opn\Ann{Ann}
\opn\soc{soc}

\title{Families of Gorenstein and almost Gorenstein rings}

\author{V. Barucci}
\address{V. Barucci - Dipartimento di Matematica - Sapienza - Universit\`a di Roma - Piazzale A. Moro 2 - 00185 Rome - Italy}
\email{barucci@mat.uniroma1.it}
\author{M. D'Anna}
\address{M. D'Anna - Dipartimento di Matematica e Informatica - Universit\`a degli Studi di Catania - Viale Andrea Doria 6 - 95125 Catania - Italy}
\email{mdanna@dmi.unict.it}
\author{F. Strazzanti}
\address{F. Strazzanti - Dipartimento di Matematica - Universit\`a di Pisa - Largo Bruno Pontecorvo 5 - 56127 Pisa - Italy}
\email{strazzanti@mail.dm.unipi.it}

\subjclass[2010]{13H10; 13A30}

\begin{document}

\begin{abstract}

\noindent Starting with a commutative ring $R$ and an ideal $I$,
it is possible to define a family of rings $\RR$, with $a,b \in
R$, as quotients of the Rees algebra $\oplus_{n \geq 0} I^nt^n$; among the rings appearing in
this family we find Nagata's idealization and amalgamated
duplication. Many properties of these rings depend only on $R$ and
$I$ and not on $a,b$; in this paper we show that the Gorenstein
and the almost Gorenstein properties are independent of $a,b$.
More precisely, we characterize when the rings in the family are
Gorenstein, complete intersection, or almost Gorenstein and we
find a formula for the type.
\end{abstract}

\keywords{Idealization; Amalgamated duplication;
Gorenstein; Almost Gorenstein; Rees algebra; complete intersection.}

\maketitle

\section*{Introduction}
Let $R$ be a commutative ring with unity and let $I \neq 0$ be a proper
ideal of $R$. In \cite{BDS} the authors introduce and study the
family of quotient rings $$R(I)_{a,b}=\R/(I^2(t^2+at+b)),$$ where
$\R$ is the Rees algebra associated with the ring $R$ with respect
to $I$ (i.e. $\R =\bigoplus_{n \geq 0}I^nt^n$) and
$(I^2(t^2+at+b))$ is the contraction to $\R$ of the ideal
generated by $t^2+at+b$ in $R[t]$.

This family provides a unified approach to Nagata's idealization
(with respect to an ideal, see \cite[pag. 2]{n}) and to
amalgamated duplication (see \cite{D'A} and \cite{DF}); they can
be both obtained as particular members of it, in particular they are
$R \ltimes I \cong R(I)_{0,0}$ and $R \Join I \cong R(I)_{0,-1}$ respectively.
This fact explains why these constructions produce rings with many
common properties; as a matter of fact, it is shown, in
\cite{BDS}, that many properties of the rings in this family
(like, e.g., Krull dimension, Noetherianity and local
Cohen-Macaulayness) do not depend on the defining polynomial. One
interesting fact about this family is that, if $R$ is a domain,
we can always find domains among its members, while the
idealization is never reduced and the amalgamated duplication is never a
domain.

In this paper we deepen the study of the rings in the family
initiated in \cite{BDS}. In particular, we characterize when these
rings are Gorenstein, complete intersection, and almost Gorenstein
and prove that these properties do not depend on the particular
member chosen in the family, but only on $R$ and $I$. The concepts
of Gorenstein ring and complete intersection ring are so prominent
that they do not need a presentation; as for the concept of almost
Gorenstein ring, we recall that it was introduced for
one-dimensional analytically unramified rings by Barucci and
Fr\"oberg in \cite{b-f}; recently this definition has been
generalized for local Cohen-Macaulay one-dimensional rings
possessing a canonical ideal (see \cite{GMP}) and successively for
rings of any Krull dimension (see \cite{GTT}). This class of rings
has been widely studied in the last years also because of its
connection with almost symmetric numerical semigroups.

The structure of the paper is the following. In the first section
we recall some properties about the family $\RR$ and complete the
characterization of its Cohen-Macaulayness. Then we prove that
$\RR$ is Gorenstein if and only if $I$ is a canonical ideal of $R$
(see Corollary \ref{gor2}); moreover, we
determine the type of $\RR$ showing that it is independent of
$a,b$ (see Theorem \ref{type0}) and, finally, we give a
characterization of the complete intersection property for $\RR$
(see Proposition \ref{ci}). In Section 2, we consider the almost
Gorenstein property of $\RR$. As for the one-dimensional case, we
give an explicit description of the canonical ideal of $\RR$ (cf.
Proposition \ref{K}) and we use it to find some characterizations
that generalize the particular cases studied in \cite{DS} and
\cite{GMP} (cf. Theorem \ref{conditions} and Corollary
\ref{conditions2}); moreover, in this case, we find a simpler
formula for the type of $\RR$ that depends only on $I$ and the
canonical module of $R$ (cf. Proposition \ref{type2}); furthermore
this formula implies that, in this case, the type of $\RR$ is odd
and included between $1$ and $2t(R)+1$, where $t(R)$ is the type
of $R$ (see Corollary \ref{type}). Finally, we prove that, also in
the higher dimensional case, the almost Gorenstein property does
not depend on $a$ and $b$ (see Proposition \ref{AGd>1}); in
particular, the results about idealization proved in \cite{GTT}
can be generalized to all the members of the family.

\section{Gorenstein property for $R(I)_{a,b}$}

We start this section by recalling some basic facts on the rings
$\RR$ proved in \cite{BDS}.

\begin{prop}\label{basics} Let $f(t)=t^2+at+b \in R[t]$ be a monic polynomial. Then
$$f(t)R[t] \cap \R=\{f(t)g(t) \ | \ g(t) \in I^2 \R\}.$$
If we denote this ideal by $(I^2f(t))$ and the quotient ring $\R/
(I^2f(t))$ by $\RR$ we have:

\smallskip
\noindent  {\rm  (1)}
$R(I)_{a,b} \cong
R \oplus I$ as $R$-module  (we will denote each element of $R(I)_{a,b}$ simply by
$r+it$, where $r \in R$ and $i \in I$);

\smallskip
\noindent {\rm  (2)} the ring extensions $R \subseteq \RR
\subseteq R[t]/(f(t))$ are both integral and, therefore, the three
rings have the same Krull dimension;

\smallskip
\noindent {\rm  (3)} let $Q$ be the total ring of  fractions of
$R(I)_{a,b}$; then each element of $Q$ is of the form $\frac{r+it}{u}$,
where $u$ is a regular element of $R$;

\smallskip
\noindent {\rm  (4)} assume that $I$ is a regular ideal, i.e.
$I$ contains a regular element; the rings $R(I)_{a,b}$ and
$R[t]/(f(t))$ have the same total ring of fractions and the
same integral closure;

\smallskip
\noindent {\rm (5)} $R$ is a Noetherian ring if and only if
$R(I)_{a,b}$ is a Noetherian ring for some $a,b \in
 R$ if and only if $R(I)_{a,b}$ is a Noetherian ring for all $a,b \in
 R$;

\smallskip
\noindent {\rm  (6)} $(R, \m)$ is local if and only if $R(I)_{a,b}$
is local. In this case the maximal ideal of $R(I)_{a,b}$ is
$M=\{m+it \ | \ m \in \m, i \in I \}$ and it is isomorphic to
$\mathfrak m \oplus I$ as $R$-module.
\end{prop}

Throughout the rest of this paper we will assume that $R$ is
Noetherian, that $I \neq 0$ is a proper ideal of $R$ and we fix all the notation above.

In order to study the Gorenstein property for $\RR$, we have to
look first at Cohen-Macaulayness (briefly CM). A weaker
formulation of the following result is Proposition 2.7 of
\cite{BDS}. For the convenience of the reader we include here the
complete proof.

\begin{prop} \label{CM}
Assume that $R$ is a local ring. The following conditions are equivalent:

\noindent  {\rm  (1)} $R$ is a CM ring and $I$ is a maximal CM $R$-module;

\noindent  {\rm  (2)} $R(I)_{a,b}$ is a CM $R$-module;

\noindent  {\rm  (3)} $R(I)_{a,b}$ is a CM ring;

\noindent  {\rm  (4)} $R$ is a CM ring and each regular $R$-sequence of $R$ is also an $R(I)_{a,b}$- regular sequence.

\end{prop}

\begin{proof}
We set $\dim R=\dim R(I)_{a,b}=d$ (cf. Proposition \ref {basics}
(2)) and observe that also $\dim _R R(I)_{a,b}=d$ as $R$-module,
because $\Ann \RR=0$.

(1) $ \Leftrightarrow$ (2): Since $\RR$ is isomorphic to $R \oplus I$ as
$R$-module, we have that $\depth \RR = \min \{ \depth R, \depth I
\}$ and so $\depth \RR=d$ (i.e. $\RR$ is a maximal CM $R$-module)
if and only if $\depth R=d$ (i.e. $R$ is a CM ring) and $\depth
I=d$. This last equality holds if and only if $\dim_R I=\depth
I=d$, i.e. $I$ is a maximal CM $R$-module.

(2) $\Leftrightarrow$ (3): We know that $R$ and $\RR$ have the same
Krull dimension. Moreover, since the extension $R\subset \RR$ is
finite, the depth of $\RR$ as $R$-module coincides with the depth
of $\RR$ as $\RR$-module (see \cite[Exercise 1.2.26]{b-h}).

(3) $ \Rightarrow$ (4): If $R$ is not a CM ring, there is a system of
parameters ${\bf x}= x_1, \dots,x_d$ of $R$ which is not an
$R$-regular sequence. Since the extension $R \subseteq \RR$ is
integral, ${\bf x}= x_1, \dots,x_d$ is also a system of parameters
of $\RR$ and it is not an $\RR$-regular sequence: in fact, if
there exists $y \in R\setminus (x_1R+\dots+x_{i-1}R)$ such that
$x_i y \in x_1R+\dots+x_{i-1}R$, then $x_i y \in
x_1\RR+\dots+x_{i-1}\RR$ and $y \notin  x_1\RR+\dots+x_{i-1}\RR$,
because  $(x_1\RR+\dots+x_{i-1}\RR) \cap R = x_1R+\dots+x_{i-1}R$.

It follows that both rings $R$ and $\RR$ are CM. To conclude we
use the fact that, in a CM ring, $\bf x$ is a regular sequence if
and only if it is part of a system of parameters (cf. \cite
[Theorem 2.1.2 (d)]{b-h}) and, as above, if $\bf x$ is part of a
system of parameters of $R$ , then it is also part of a system of
parameters of $\RR$.

(4) $ \Rightarrow$ (2): We know that there exists an $R$-regular
sequence of $R$ of length $d$ . It is also an $\RR$- regular
sequence and so $\RR$ is a CM $R$-module.

\end{proof}

We recall (following \cite[Sections 21.1 and 21.3]{E}) that a
canonical module of a zero-dimensional local ring is defined as
the injective hull of its residue class field; if $R$ is a local
CM ring of dimension $d>0$, then a finitely generated $R$-module
$\omega_R$ is a canonical module of $R$ if there
exists a non-zerodivisor $x \in R$ such that $\omega_R/x\omega_R$
is a canonical module of $R/(x)$. If $R$ has a canonical module
$\omega_R$ and this is isomorphic to an ideal of $R$, we say that
$\omega_R$ is a canonical ideal of $R$. It is well known that a CM
local ring $R$ has a canonical module if and only if it is the
homomorphic image of a Gorenstein local ring (cf. e.g.
\cite[Proposition 3.3.6]{b-h}) and the canonical module is
isomorphic to an ideal $I$ if and only if $R$ is generically
Gorenstein, i.e. $R_\p$ is Gorenstein for each minimal prime $\p$
of $R$ (see \cite[Exercise 21.18]{E}).

The authors proved in \cite[Corollary 3.3]{BDS} that, if $R$ is a
one-dimensional Noetherian local ring and $I$ is a regular ideal of
$R$, then $\RR$ is Gorenstein if and only if $I$ is a canonical ideal of $R$.

The main goal of this section is to generalize this result to any
dimension $d\geq 0$. More generally in the next theorem
we compute the type of $\RR$ generalizing \cite[Theorem 3.2]{BDS}.

\begin{thm}\label{type0} Let $(R,\m)$ be a local CM ring of Krull dimension $d \geq 1$ and let $I$ be
a regular ideal and a maximal CM $R$-module. Then the CM type of
$\RR$ is
$$
t(\RR)=\lambda_R\left(\frac{(J:\m)\cap(JI:I)}{J}\right)+\lambda_R\left(\frac{(JI:\m)}{JI}\right),
$$
where $\lambda_R(.)$ denotes the length of an $R$-module and
$J=(x_1,x_2,\dots,x_d)$ is an ideal of $R$ generated by an
$R$-regular sequence.

In particular, the type of $\RR$ is independent of $a,b$.
\end{thm}

\begin{proof}
Let $M$ be the maximal ideal of $\RR$. It is well known that
$$\aligned
t(\RR)=&\lambda_{\RR}\left(
\text{Ext}^d_{\RR}(\RR/M,\RR)\right)=\\
=&\lambda_{\RR}\left(\Hom_{\RR}(\RR/M,\RR/H)\right)=\lambda_{\RR}\left(\frac{H:M}{H}\right)
\endaligned
$$
for any ideal $H$ generated by an $\RR$-regular sequence (see \cite[Theorem 3.1 (ii)]{R}).
By Proposition \ref{CM} we can choose $H$ generated by an $R$-regular
sequence ${\bf x}= x_1, \dots,x_d$. This means that
$H=J\RR=\{j_1+ij_2t | \ j_1,j_2 \in J, i \in I, t^2=-at-b\}$,
where $J$ is the ideal of $R$ generated by ${\bf x}$.
Moreover, since $(H:M)/H$ is annihilated by $\m$, its length as $\RR$-module
coincides with its length as $R$-module (see \cite[Remark
2.2]{BDS}). Hence
$$t(\RR)=\lambda_R\left(\frac{(J\RR:M)}{J\RR}\right).$$
We want to show that
$$(J\RR:M)=
\left\{\frac{r}{s}+\frac{i}{s} \ t\ ; \ \frac{i}{s} \in (JI:\m),
\frac{r}{s} \in (JI:I) \cap (J:\m)\right\}.$$ Since $I$ is a
regular ideal, a generic element of $Q(R(I)_{a,b})$ is of the form
$r/s+(i/s)t$, where $r,s \in R, i \in I$ and $s$ is regular (cf.
Proposition \ref{basics}(3)). Therefore it is an element of
$(J\RR:M)$ if and only if
$$
\begin{array}{ll}
(r/s+(i/s)t)(m+jt)&=rm/s+(im/s)t+(rj/s)t+(ij/s)t^2=\\
&=rm/s-ijb/s+(im/s+rj/s-ija/s)t
\end{array}
$$
is an element of $J\RR$, for any $m \in \m$ and for any $j
\in I$; that is $(rm/s-ijb/s) \in J$ and $(im/s + rj/s -ija/s) \in
JI$.

Suppose that $r/s+(i/s)t \in (J\RR:M)$; in
particular, if $j=0$ we have $rm/s \in J$ and $im/s \in JI$, that
is $r/s \in (J:\m)$ and $i/s \in (JI:\m)$. Moreover, since $ja \in
I \subseteq \m$ and $i/s \in (JI:\m)$, we have $im/s,ija/s \in
JI$, hence $rj/s \in JI$ for any $j \in I$ and then $r/s \in
(JI:I)$.

Conversely, suppose that $i/s \in (JI:\m)$ and $r/s \in (JI:I)
\cap (J:\m)$. Then $rm/s-ijb/s \in J+JI=J$ and $im/s +rj/s - ija/s
\in JI+JI+JI=JI$, consequently $r/s+(i/s)t \in (J\RR:M)$.

Now it is straightforward to see that the homomorphism of
$R$-modules
$$
(J\RR:M) \ \longrightarrow \
\left(\frac{(J:\m)\cap(JI:I)}{J}\right) \times
\left(\frac{(JI:\m)}{JI}\right)
$$
defined by $r/s+(i/s)t \mapsto (r/s+J, i/s+JI)$ is surjective and
its kernel is $J\RR$. The thesis follows immediately.
\end{proof}

\begin{cor}\label{gor2} Let $R$ be a local ring of dimension $d\geq 1$ and let $I$ be a
  regular ideal of $R$. Then,
for every $a,b \in R$, the ring $\RR$ is Gorenstein if and only if
$R$ is a CM ring and $I$ is a canonical ideal of $R$.
\end{cor}

\begin{proof} Under our hypotheses, it is well known that the
idealization (see \cite{r}) and the duplication (see \cite{D'A}
and \cite{Sh}) produce a Gorenstein ring if and only if R is CM and $I$ is a
canonical ideal. Since a CM ring is Gorenstein if and only if its CM type is one,
the thesis follows immediately by Theorem \ref {type0}.
\end{proof}

We notice that, if $I$ is a canonical ideal of $R$,
we can apply a result of Eisenbud
(stated and proved in \cite{D'A}) to prove one
direction of the above result, i.e.  that $R(I)_{a,b}$ is
Gorenstein for every $a, b \in R$.


\begin{cor} \label{regular} Let $R$ be a regular local ring. The ring $\RR$ is CM if and only if it is Gorenstein.
\end{cor}

\begin{proof} An ideal $I$ of a regular local ring is a maximal CM module if
and only if it is a principal ideal by the Auslander-Buchsbaum
formula (cf. \cite[Theorem 1.3.3]{b-h}) if and only if it is a
canonical ideal, because a regular local ring is Gorenstein.
Therefore it is enough to apply Proposition \ref{CM} and Corollary
\ref{gor2}.
\end{proof}

Notice that, if $R$ is zero-dimensional, its canonical module is
isomorphic to an ideal if and only if $R$ is Gorenstein and, in
this case, we have $\omega_R \cong R$; in any case, $\omega_R$ is
never isomorphic to a proper ideal of $R$ and therefore the next
theorem is not surprising.

\begin{thm}

Let $(R,\m)$ be a local Artinian ring.
Then $\RR$ is not Gorenstein.

\end{thm}

\begin{proof}
Since $\RR$ is an Artinian ring, it is Gorenstein if and only if its socle,
$\soc \RR =(0:_{\RR}M)$, is a $k$-vector space of dimension one, where
$k=\RR/M \cong R/\m$. We have that
$r+it \in \soc \RR$ if and only if
$$
\begin{cases}
rm-ijb=0, \\
rj+mi-aij =0
\end{cases}
$$
for any $m+jt \in M$ (i.e. for any $m \in \m$ and any $j \in I$).
In particular, if $j=0$ we get $r \in \soc R$ and $i \in I \cap
\soc R$; thus
$$
\soc \RR\subseteq \{r+it \ | \ r \in \soc R, i \in I \cap \soc R \}.
$$
It is straightforward to check the opposite inclusion, so we have an equality.
We claim that $I \cap \soc R \neq (0)$. Indeed, if $0 \neq x \in I$,
we have $x\m^n=(0)$ for some $n \in \mathbb{N}$, because $\m$ is nilpotent by artinianity.
We can assume that $x \m^{n-1} \neq (0)$ and clearly $x \m^{n-1} \subseteq I \cap \soc R$.

Consequently, if $0 \neq i \in I \cap \soc R$, we have that $i$
and $it$ are elements of $\soc \RR$ and they are linearly
independent on $k$; hence $\RR$ is not a Gorenstein ring.
\end{proof}

We end this section studying when $\RR$ is a complete intersection
(briefly, c.i.). We recall that, following \cite[Section
18.5]{E}, a local ring $R$ is a c.i. if its completion
with respect to the $\m$-adic topology $\hat R$ can be written as a
regular local ring modulo a regular sequence.

\begin{rem}\label{completion}
\rm Assume that $(R,\m)$ is local; in this case we know that also
$\RR$ is local, with maximal ideal $M=\m\oplus I$ (see Proposition
\ref{basics}(6)). Since the powers of $M$ are, as $R$-modules,
$M^n=\m^n\oplus \m^{n-1}I$ (see the proof of \cite[Proposition 2.3]{BDS}),
it is straightforward that the $M$-adic topology on $\RR$ coincides
with the $\m$-adic topology induced by the structure of $\RR$ as
$R$-module. Hence, as $R$-module, $\widehat {\RR} \cong \hat R\oplus
\hat I$.

Since we are supposing $R$ to be Noetherian, we can assume that
$R\subset \hat R$ and thus $a,b \in \hat R$; now it is clear that
$\widehat {\RR} \cong \hat R(\hat I)_{a,b}$.
\end{rem}

\begin{prop} \label{ci}
Let $R$ be a local ring and let $I$ be a  regular
ideal of $R$. The ring $\RR$ is a c.i. if and only if $R$ is a
c.i. and $I$ is a canonical ideal. In particular the property of
being a c.i. is independent of the choice of $a$ and $b$.
\end{prop}

\begin{proof}
Let $I$ be minimally generated by $i_1, \dots, i_p$. By Cohen's
structure theorem we have that $\hat R \cong S/J$, where $S$ is a
complete regular local ring. It follows that the ring
$\widehat{\RR}$ can be presented as $S[[y_1,\dots,y_p]]/\ker
\varphi$, where $\varphi: S[[y_1,\dots,y_p]]\rightarrow \widehat
\RR$ is defined by $\varphi(s)=s+J$ and $\varphi(y_h)=i_h t$, for
every $h=i,\dots, p$. Notice that $S[[y_1,\dots,y_p]]$ is again
a regular local ring.

Since $(i_h t)^2=-ai_h^2 t-bi_h^2$, with $ai_h, bi_h^2 \in \hat
R$, if we choose $\alpha_h, \beta_h \in S$ such that
$\varphi(\alpha_h)=ai_h$ and $\varphi(\beta_h)=bi_h^2$, then $\ker
\varphi$ contains the elements of the form $F_h:=y_h^2+\alpha_h
y_h+\beta_h$. Hence $\ker \varphi \supseteq J+(F_1,\dots,F_p)$.
For every index $h$, an element of the form $F_h$ is necessary as
a generator of $\ker \varphi$, since it contains a pure power of
$y_h$ of the lowest possible degree. Moreover, $\ker \varphi \cap
S= J$, since the restriction of $\varphi$ to $S$ gives the
presentation of $\hat R$. It follows that $\mu(\ker \varphi)$
(i.e. the cardinality of a minimal set of generators of $\ker
\varphi$) is bigger than or equal to $\mu(J)+p$.

Assume that $\RR$ is a c.i.; this means that $\dim S+p-\dim
\widehat \RR=\mu(\ker \varphi)$. Hence we have the following chain
of inequalities:
$$
\dim S+p-\dim \widehat \RR=\mu(\ker \varphi) \geq \mu(J)+p \geq
\dim S-\dim \hat R+p.
$$
Since $\dim \hat R=\dim \widehat \RR$, all the above inequalities
are equalities and, in particular, $\mu(J)=\dim S-\dim \hat R$,
i.e. $R$ is a c.i..

Moreover, since $\RR$ is a c.i., it is Gorenstein and
$I$ has to be a canonical ideal of $R$ by Corollary \ref{gor2}.

Conversely, assume that $R$ is a c.i. and that $I$ is a canonical
ideal of $R$. We have that $\mu(I)=1$, since it equals the type of
$\hat R$, which is Gorenstein. Using the above notation we have $\ker
\varphi \supseteq J+(F_1)$. The reverse inclusion is also true,
since, if $g(y_1) \in \ker \varphi$, its class modulo $J+(F_1)$ is
of the form $g_0+g_1y_1$ (with $g_0,g_1 \in S$) and it belongs to
$\ker \varphi$ if and only if $g_0 \in J$ and
$\varphi(g_1)i_1t=0$; the last equality, since $i_1$ is a
non-zerodivisor, implies that also $g_1 \in J$. This proves that
$\mu(\ker \varphi)=\mu(J)+1$; since $\mu(J)=\dim S -\dim \hat R$,
also $\RR$ is a c.i..
\end{proof}

\section{Almost Gorenstein property for $R(I)_{a,b}$}

Let $(R,\m)$ be a local one-dimensional Cohen-Macaulay  ring. We
say that $R$ is an {\it almost Gorenstein} ring if it has a
canonical module $\omega_R$ which is isomorphic to a fractional
ideal of $R$ such that
$$R \subseteq \omega_R \subseteq (\m:\m).$$ This definition
generalizes the first one given in \cite{b-f} for one-dimensional
ana\-lytically unramified rings and it is equivalent to the
definition given in \cite{GMP} if $R/\m$ is infinite, since, in
this case, we can assume that $R \subseteq \omega_R \subseteq
\overline R$ (see \cite[Theorem 3.11]{GMP}).
Thus for a local one-dimensional almost Gorenstein ring we have an
exact sequence of $R$-modules
$$ 0 \rightarrow R \rightarrow \om \rightarrow \om/R \rightarrow 0$$
with $\m\om \subseteq \m$ or,
equivalently, $\m \om \subseteq R$, i.e. $\m (\om/R)=0$.

Following \cite{GTT}, a CM local ring $(R,\m)$ of any Krull
dimension $d$ possessing a canonical module $\om$ is defined to
be {\it almost Gorenstein} if there exists an exact sequence of
$R$-modules
$$ 0 \rightarrow R  \rightarrow \om  \rightarrow C  \rightarrow
0$$ such that $\mu_R(C)=e_{\m}^0(C)$, where $\mu_R(C)=
\lambda_R(C/{\m C})$ is the number of generators of $C$ and
$e_{\m}^0(C)$ is the multiplicity of $C$ with respect to $\m$.  It
turns out that dim$_RC=d-1$.\
Thus, in dimension $d$, the condition $\m C=0$, given in dimension one,
becomes $\m C=(f_1, \dots, f_{d-1})C$,
for some $f_1, \dots, f_{d-1} \in \m$,  i.e.
$\mu_R(C)= \lambda_R (C/\m C)= \lambda_R (C/ (f_1, \dots, f_{d-1})C)  =e_{\m}^0(C)$.

Moreover, if $R$ is one-dimensional and satisfies this general
definition, we may assume that the canonical module $\om$ is a
fractional ideal of $R$, i.e. that $\om$ is a canonical ideal of
$R$, because the total ring of fractions of $R$ turns out to be a
Gorenstein ring (cf. \cite[Lemma 3.1,(1) and Remark 3.2]{GTT}). If
we also assume that $R/\m$ is infinite, the definition of
one-dimensional almost Gorenstein ring given in the beginning of this  section, which we
adopt, is equivalent to that given in \cite{GTT}, as proved in
\cite[Proposition 3.4]{GTT}.

We finally recall that, if $d=0$, a ring is almost Gorenstein
if and only if it is Gorenstein.

The goal of this section is to study when $\RR$ is almost
Gorenstein and to prove that  this property is independent also of
the choice of $a,b \in R$. We first study the one-dimensional
case, giving an explicit description of the canonical ideal of
$\RR$ and some constructive methods to get almost Gorenstein rings;
then, we study the case of dimension $d>1$.

Throughout this section we assume that $\RR$ is a CM local ring;
we recall that it is equivalent to require that $R$ is CM and
local and $I$ is  a maximal CM module (cf. Proposition \ref{CM});
we will also assume that $R/\m$ is infinite.

\subsection{The one-dimensional case}

Let $(R,\m)$ be a one-dimensional Cohen-Macaulay local ring and
$I$ be an $\m$-primary ideal of $R$. We further assume throughout
this subsection that $R$ has a canonical ideal $\omega_R$ that
is a fractional ideal such that $R \subseteq \omega_R \subseteq
\overline{R}$.

Let $H$ be a fractional ideal of $R$; since by definition there
exists a regular element $y \in R$ such that $yH=J \subset R$, we
can consider a minimal reduction $xR$ of $J$ and, with a slight
abuse of terminology, we call $xy^{-1}R$ a minimal reduction of
$H$, where now $xy^{-1}\in Q(R)$, the total ring of fractions of
$R$.

Let $zR$ (with $z \in Q(R)$) be a minimal reduction of $(\omega_R
: I)$ and let us fix this notation for the whole current
subsection; note that in this case $z$ has to be an invertible
element of $Q(R)$: in fact, if $z=xy^{-1}$ as above, $y \in
y(\om:I)$, so this is a regular ideal and a minimal reduction of a
regular ideal has to be generated by a non-zerodivisor.

The inclusion $R \subseteq R(I)_{a,b}$ is a local homomorphism and
$\RR$ is a finite $R$-module, hence the canonical module of $R(I)_{a,b}$ is
$\Hom_R(R(I)_{a,b},\omega_R)$ (by \cite[Theorem 3.3.7 (b)]{b-h}),
where the structure of $R(I)_{a,b}$-module is given by
$((r+it) \varphi)(s+jt)=\varphi((r+it)(s+jt))$, for each $\varphi \in
\Hom_R(R(I)_{a,b},\omega_R)$.

Our first goal is to give an explicit description of a canonical
ideal $K$ of $\RR$ such that $\RR \subseteq K \subseteq
\overline{\RR}$.

 Clearly, as $R$-modules,
$$
\omega_{R(I)_{a,b}} \cong {\rm Hom}_R(R(I)_{a,b},\omega_R) \cong
{\rm Hom}_R(R \oplus I, \omega_R) \cong
$$
$$
\cong {\rm Hom}_R(R,\omega_R) \oplus {\rm Hom}_R(I,\omega_R) \cong
\omega_R \oplus (\omega_R : I) \cong \frac{1}{z}(\omega_R : I)
\oplus \frac{1}{z} \omega_R.
$$

We want to see that $\frac{1}{z}(\omega_R : I) \oplus \frac{1}{z}
\omega_R$ is also an $R(I)_{a,b}$-module isomorphic to
$\omega_{R(I)_{a,b}}$. More precisely, we define
$$K= \left\{ \frac{x}{z}+ \frac{y}{z}t | \ x \in (\omega_R : I), y \in \omega_R \right\}$$
and, given  $(r+it) \in R(I)_{a,b}$ and
$(\frac{x}{z}+\frac{y}{z}t) \in K$, we set
$$
(r+it) \left( \frac{x}{z} +\frac{y}{z}t \right)= \left(
\frac{rx}{z}-\frac{biy}{z}+
\left(\frac{ry}{z}+\frac{ix}{z}-\frac{aiy}{z}\right)t\right) \in
K;
$$
it is easy to see that, in this way, we define an
$R(I)_{a,b}$-module.

\begin{prop}\label{K}
The $\RR$-module $K$, defined above, is a canonical ideal of $\RR$
such that $\RR \subseteq K \subseteq \overline{\RR}$.
\end{prop}

\begin{proof}
Consider the map $\varphi: K \rightarrow {\rm
Hom}_R(R(I)_{a,b},\omega_R)$ that associates with $(\frac{x}{z}
+\frac{y}{z}t)$ the homomorphism
$f_{(\frac{x}{z}+\frac{y}{z}t)}:(s+jt) \mapsto (xj+y(s-ja))$. It
is enough to prove that this is an isomorphism of
$R(I)_{a,b}$-modules. Clearly $\varphi$ is well defined. Let
$(r+it),(s+jt) \in R(I)_{a,b}$ and $(\frac{x}{z}+\frac{y}{z}t) \in
K$, one has
\begin{equation*}
\begin{split}
&\left((r+it)\varphi\left(\frac{x}{z}+\frac{y}{z}t\right)\right)(s+jt)=(r+it)f_{(\frac{x}{z}+\frac{y}{z}t)}(s+jt)= \\
&=f_{(\frac{x}{z}+\frac{y}{z}t)}((r+it)(s+jt)) =f_{(\frac{x}{z}+\frac{y}{z}t)}(rs-bij+(rj+is-aij)t)= \\
&=xrj + xis -aijx +yrs -bijy - arjy -aisy + a^2ijy=\\
&=f_{(\frac{rx-biy}{z}+\frac{ix+ry-aiy}{z}t)}(s+jt)=\varphi\left((r+it)\left(\frac{x}{z}+\frac{y}{z}t\right)\right)(s+jt).
\end{split}
\end{equation*}
This proves that $\varphi$ is an homomorphism of
$R(I)_{a,b}$-modules. Moreover, if
$f_{(\frac{x}{z}+\frac{y}{z}t)}(s+jt)=0$ for any $(s+jt) \in
R(I)_{a,b}$, chosen $\lambda \in I$ regular, one has

$$
  \begin{cases}
  y=f_{(\frac{x}{z}+\frac{y}{z}t)}(1)=0 \\
  \lambda x=f_{(\frac{x}{z}+\frac{y}{z}t)}(\lambda a + \lambda t)=0
  \end{cases}
$$

then $(x,y)=(0,0)$ and therefore $\varphi$ is injective.

As for the surjectivity, consider $g \in {\rm
Hom}_R(R(I)_{a,b},\omega_R)$. Let $\lambda \in I$ be a regular
element and set

$$
  \begin{cases}
  x=\frac{g(\lambda t)}{\lambda} + g(a) \\
  y=g(1)
  \end{cases}
$$
Clearly $y \in \omega_R$ and we claim that $x \in (\omega_R :I)$;
in fact, if $i \in I$,
$$
ix = \frac{ig(\lambda t)}{\lambda}+ig(a)=\frac{\lambda
g(it)}{\lambda} + g(ai)=g(ai+it) \in \omega_R.
$$
Hence $\frac{x}{z}+\frac{y}{z}t \in K$. Finally, for any $s+jt \in
R(I)_{a,b}$, one has
$$
f_{(\frac{x}{z}+\frac{y}{z}t)}(s+jt)=xj+y(s-ja)=\frac{g(\lambda
t)}{\lambda}j+g(aj)+g(s)-g(aj)=
$$
$$
=\frac{\lambda g(jt)}{\lambda} + g(s)=g(s+jt)
$$
and consequently $\varphi$ is surjective.

We recall that, by Corollary $1.8$ of \cite{BDS}, the integral closure
of $R(I)_{a,b}$ contains the ring
$\overline{R}[t]/(t^2+at+b)= \{r_1+r_2t| \ r_1,r_2 \in
\overline{R}, t^2=-at-b \}$.

One has $R \subseteq \frac{1}{z}(\omega_R : I)$ and $I \subseteq
\frac{1}{z} \omega_R$, because $z \in (\omega_R :I)$, thus $\RR
\subseteq K$. Moreover $\omega_R \subseteq (\omega_R : I)
\subseteq z\overline{R}$, since $z$ is a minimal reduction of
$(\omega_R : I)$ (see e.g. \cite[Proposition 16]{b-f}). Hence
$$R(I)_{a,b} \subseteq K \subseteq \overline{R}[t]/(t^2+at+b)
\subseteq \overline{R(I)_{a,b}}.$$ Finally $K$ is a fractional
ideal of $\RR$. In fact we can choose two regular elements
$i \in I$ and $r\in R$, such that $r \om \subseteq R$;
hence $riz \in R \subseteq \RR$ is such that $riz K \subseteq \RR$.
\end{proof}

The next lemma is proved, in a different way, in the proof of
\cite[Proposition 6.1]{GMP}. If $\mathfrak{a}$ is a fractional
ideal of $R$, we denote its dual, $(\omega_R : \mathfrak{a})$, by
$\mathfrak{a}^{\vee}$.

\begin{lem} \label{Goto} Let $\mathfrak{a}, \mathfrak{b},$ and $\mathfrak{c}$
be fractional ideals of $R$ and let $z$ be a reduction of $I^{\vee}$. The following statements hold. \\
\noindent {\rm (1)} $\mathfrak{ab} \subseteq \mathfrak{c}$ if and only if $\mathfrak{ac}^{\vee} \subseteq \mathfrak{b}^{\vee}$. \\
\noindent{\rm (2)} $\m I^{\vee} \subseteq zR$ if and only if $\m \omega_R \subseteq zI$. \\
\noindent{\rm (3)} $II^{\vee}=zI$if and only if $zI^{\vee}=(I^{\vee})^2$.
\end{lem}

\begin{proof}
As for the first point we have that
$$\mathfrak{ab} \subseteq \mathfrak{c} \Leftrightarrow \mathfrak{c}^{\vee}
\subseteq (\mathfrak{ab})^{\vee} \Leftrightarrow
\mathfrak{c}^{\vee} \subseteq (\mathfrak{b}^{\vee}:\mathfrak{a})
\Leftrightarrow \mathfrak{a}\mathfrak{c}^{\vee} \subseteq
\mathfrak{b}^{\vee}.
$$

The second point follows applying this to $\mathfrak{a}=\m$,
$\mathfrak{b}=I^{\vee}$, $\mathfrak{c}=zR$. In the same way we get
the last point in the particular case $\mathfrak{a}=I^{\vee}$,
$\mathfrak{b}=I$ and $\mathfrak{c}=zI$, because the other
inclusions are trivial.
\end{proof}

We can see that Proposition 6.1 of \cite{GMP}, proved for the
idealization $R \ltimes I \cong R(I)_{0,0}$, holds also for arbitrary
$a$ and $b$.

\begin{thm}\label{conditions}
The ring $R(I)_{a,b}$ is almost Gorenstein if and only if
$II^{\vee}=zI$ and $z\m=\m I^{\vee}$. In particular, almost
Gorensteinness does not depend on $a$ and $b$.
\end{thm}

\begin{proof}
By Proposition \ref{K}, we have the canonical ideal $K$ of $\RR$
defined above. Let $M$ be the maximal ideal of $R(I)_{a,b}$.
$R(I)_{a,b}$ is almost Gorenstein if and only if $MK \subseteq M$
or, equivalently, $M zK \subseteq zM$. Given $(m+it) \in M$ and
$(x+yt) \in zK$ (i.e. $m \in \m$, $i \in I$, $x \in I^{\vee}$,  $y
\in \om$), the latter condition means that  $(m+it)(x+yt)=
mx-biy+(my+ix-aiy)t \in zM$, that is
$$
\begin{cases}
mx-biy \in z\m \\
my+ix-aiy \in zI.
\end{cases}
$$

Suppose now that $R(I)_{a,b}$ is almost Gorenstein. If we choose
$i=0$, the first equation becomes $\m I^{\vee} \subseteq z \m$,
i.e. $\m I^{\vee} = z \m$. Moreover, if in the second equation we
set $y=0$, we get $II^{\vee} \subseteq zI$, i.e. $II^{\vee}= zI$.

Conversely, if the conditions of the statement hold, in light
of the previous lemma we have
\begin{equation*}
\begin{split}
&mx-biy \in  \m I^{\vee}+I\omega_R \subseteq z\m +\m \omega_R
\subseteq z\m+zI \subseteq z\m + z \m = z \m \\
&my+ix-aiy \in  \m \omega_R + II^{\vee}+I \omega_R \subseteq zI
+zI + \m \omega_R \subseteq zI.
\end{split}
\end{equation*}
\end{proof}

%
%
%


Theorem \ref{conditions} may have the following equivalent
formulation:
\begin{cor}\label{conditions2}
The ring $R(I)_{a,b}$ is almost Gorenstein if and only if
 $z^{-1}I^{\vee}$ is a ring and $R \subseteq z^{-1}I^{\vee} \subseteq (\m:\m)$.
 \end{cor}

\begin{proof}
By Lemma \ref{Goto}, the condition $II^{\vee}=zI$ is equivalent to
$zI^{\vee}=(I^{\vee})^2$ and we note that this happens if and only
if $z^{-1}I^{\vee}$ is a ring. Indeed $zI^{\vee}=(I^{\vee})^2$ if
and only if for any $x,y \in I^{\vee}$ one has $xy \in zI^{\vee}$,
that is  equivalent to $z^{-1}xz^{-1}y \in z^{-1}I^{\vee}$ for any
$z^{-1}x,z^{-1}y \in z^{-1}I^{\vee}$, i.e. $z^{-1}I^{\vee}$ is a
ring.

Furthermore, the condition $z\m=\m I^{\vee}$ is equivalent to
 $(\m :\m) \supseteq z^{-1}I^{\vee}$, i.e. $(z\m :\m) \supseteq I^{\vee}$,
because we always have $\m I^{\vee} \supseteq z\m$. Finally, since
$zI \subseteq \om$, we have $R= (\om : \om) \subseteq (\om : zI) =
z^{-1}I^{\vee}$.

\end{proof}

Corollary \ref{conditions2} allows us to construct a large class of
one-dimensional almost Gorenstein rings. In fact, let $A$ be an
overring of $R$, $A \subseteq (\m : \m)$. Then $A^{\vee}= (\om :
A)$ is a fractional ideal of $R$. Let $r \in R$ be a regular
element such that $rA^{\vee} \subseteq R$ and set $I:= rA^{\vee}$.
It is easy to check that $I$ satisfies the conditions of
Corollary \ref {conditions2}, in fact a minimal reduction of
$I^{\vee}=r^{-1}A$ is $z=r^{-1}$ and $z^{-1}I^{\vee}=rr^{-1}A=A$.

If $A=R$, then $A^{\vee}= (\om :A)=\om$ and any integral ideal
$I=rA^{\vee}$ is a canonical ideal, giving $\RR$ Gorenstein.

If $A= (\m:\m)$, then $A^{\vee}= (\om :(\m:\m))$ and any integral
ideal of the form $r(\om :(\m:\m))$ gives $\RR$ almost Gorenstein
(cf. \cite [Corollary 6.2]{GMP}).

In particular, if $R$ is Gorenstein then there are not proper
overrings between $R$ and $(\m:\m)$. It follows that $\RR$ is
almost Gorenstein if and only if either $I=rR^{\vee}=rR$ or
$I=r(R:(\m:\m))=r\m$  (cf. \cite [Corollary 6.4]{GMP}).

\begin{ex} \label{ex1}
{\rm   Consider $R:=k[[X^4,X^5,X^{11}]]$, where $k$ is a field.
In this case  $(\m:\m)=k[[X^4,X^5,X^6,X^7]]$. If we choose the overring
$A=k[[X^4,X^5,X^7]]$ of $R$, then $A^{\vee}$ is the fractional
ideal $(X,X^4)$ of $R$ and taking for example $r=X^4$, we get
$I=X^4A^{\vee}=(X^5,X^8)$. Thus, for any choice of $a,b \in R$, we
obtain that $\RR$ is an almost Gorenstein ring.

We point out that, if $I$ and $J$ are two isomorphic ideals of
$R$, $\RR$ and $R(J)_{a,b}$ are not necessarily isomorphic. For
example, if we choose the ideal $I$ above and
$J=X^7A^{\vee}=(X^8,X^{11})$, with $a=0$, $b=-X^5$, we get
$R(I)_{0,-X^5}=k[[T_1]]$ and $R(J)_{0,-X^5}=k[[T_2]]$, with $T_1=
\langle 8,10,15,21,22 \rangle$ and $T_2=
\langle8,10,21,22,27\rangle$ (cf. \cite[Theorem 3.4]{BDS}).
However, if $J=xI$, $\RR$ is almost Gorenstein if and only if
$R(J)_{a,b}$ is almost Gorenstein: in fact, if $z$ is a reduction
of $I^{\vee}$, then $x^{-1}z$ is a reduction of $J^{\vee}$;
moreover $ (x^{-1} z)^{-1} J^{\vee}= x
z^{-1}J^{\vee}=z^{-1}I^{\vee}$, so, by previous corollary, the
conditions required to be almost Gorenstein coincide for both rings. }
\end{ex}

If $R$ is a numerical semigroup ring or an algebroid branch, it is
possible to get information about $\RR$ by studying a numerical
semigroup, called {\it numerical duplication} (see \cite[Theorems
3.4 and 3.6]{BDS}). In numerical semigroup theory, the
corresponding concept of almost Gorenstein ring is the notion of
almost symmetric semigroup and, in this context, Corollary
\ref{conditions2} generalizes Theorem 4.3 of \cite{DS}. Moreover,
in this case a simple formula is known for the type of the
numerical duplication (see \cite[Proposition 4.8]{DS}). The next
proposition generalizes this result, giving a formula for
$t(\RR)$, the CM type of $\RR$.

 \begin{prop} \label{type2}
Suppose that $\RR$ is almost Gorenstein, then
$$
t(\RR)= 2\lambda_R \left(\frac{z^{-1}I^{\vee}}{R}\right)+1 =
2\lambda_R \left(\frac{\omega_R}{zI}\right)+1.
$$
\end{prop}

\begin{proof} First recall that the CM type of $\RR$ is
$$
t(\RR)=\lambda_R\left(\frac{(I:I)\cap(R:\m)}{R}\right) +
\lambda_R\left(\frac{(I:\m)}{I}\right)
$$
by \cite[Theorem 3.2]{BDS}.

We note that $(I:I) \subseteq (R:\m)$. Indeed it is easy to see
that $z(I:I) \subseteq (\omega_R : I) \subseteq z(R:\m)$, because
$\m I^{\vee} =z\m$ by Theorem \ref{conditions}.

Moreover $(I:I)=((\omega_R : (\omega_R :I)):I)=((\omega_R :
I^{\vee}) : I)=(\omega_R:II^{\vee})=z^{-1}I^{\vee}$, because
$zI=II^{\vee}$ again by Theorem \ref{conditions}.

Finally, to conclude the proof, it is enough to show that
$\lambda_R((I:\m)/I)=\lambda_R(z^{-1}I^{\vee}/\m)$. This holds
because
\begin{equation*}
\begin{split}
\lambda_R
\left(\frac{(I:\m)}{I}\right)&=\lambda_R\left(\frac{((\omega_R:(\omega_R:I)):\m)}{I}\right)=
\lambda_R\left(\frac{(\omega_R: \m I^{\vee})}{I}\right)= \\
&=\lambda_R \left(\frac{(\omega_R:I)}{ \m I^{\vee}}\right)=
\lambda_R \left(\frac{I^{\vee}}{z\m}\right)= \lambda_R\left(\frac{z^{-1}I^{\vee}}{\m}\right),
\end{split}
\end{equation*}
where, since $R(I)_{a,b}$ is almost Gorenstein, we used  $\m
I^{\vee}=z \m$ (cf. Theorem \ref{conditions}).
\end{proof}

By Corollary \ref{regular}, if $R$ is a DVR and $\RR$ is almost Gorenstein,
it follows that $\RR$ is Gorenstein, i.e. has type $1$.
On the other hand if we assume that $R$ is not a DVR, the type of $R$ is the length
of $(\m: \m)/R$ and then, since in the almost Gorenstein case
$z^{-1}I^{\vee} \subseteq (\m : \m)$ (cf. Corollary
\ref{conditions2}), the previous proposition implies the
following.

\begin{cor} \label{type}
If $\RR$ is almost Gorenstein, its type
is odd and $1 \leq t(\RR) \leq 2t(R)+1$.
\end{cor}

\begin{ex}{\rm In Example \ref{ex1} we get $t(\RR)=2\lambda_R(A/R)+1=3$.
Observe that in this example $t(R)=2$
and all the odd values $t$, $1 \leq t \leq 5=2t(R)+1$
can be realized for $t(\RR)$. In fact for example for $I=(X^4,X^5)$,
which is a canonical ideal of $R$, we get $t(\RR)=1$ and for
$I=(X^4,X^5,X^6)=X^4(\om :(\m:\m))$ we get $t(\RR)=5$. }
\end{ex}

Since the almost Gorensteinness does not depend on $a$ and $b$,
from \cite[Theorem 6.5]{GMP} we get the following proposition, of
which we include also a simple proof.

\begin{prop} The ring $R$ is almost Gorenstein if and only if $R(\m)_{a,b}$ is almost Gorenstein.
In this case, if $R$ is not a DVR, the type of $R(\m)_{a,b}$ is $2t(R)+1$.
\end{prop}

\begin{proof}We have that $R$ is a DVR if and only if $R(\m)_{a,b}$
is Gorenstein (cf. Corollary \ref{gor2}). Thus we can exclude this case.
If $R$ is almost Gorenstein and not a DVR, then $(\om:
\m)=(\m:\m)$: in fact,
$\lambda_R((\om:\m)/\om)=\lambda_R(R/\m)=1$; moreover, denoting by
$t(R)$ the CM type of $R$, we obtain the following chain of
equalities
$t(R)+1=\lambda_R((\m:\m)/\m)=\lambda_R((\m:\m)/\om)+\lambda_R(\om/\m)=
\lambda_R((\m:\m)/\om)+\lambda_R(\om/\m\om)=\lambda_R((\m:\m)/\om)+t(R)$,
that implies $\lambda_R((\m:\m)/\om)=1$ (cf.
\cite[Definition/Proposition 20]{b-f}). So if
 $I=\m$, then $z=1$, $z^{-1}I^{\vee}= (\m:\m)$   and, by Corollary
\ref{conditions2}, $R(\m)_{a,b}$ is almost Gorenstein.

Conversely, if $R(\m)_{a,b}$  is almost Gorenstein but not
Gorenstein, then $z^{-1} (\om:\m) \subseteq (\m:\m)$ (Corollary
\ref{conditions2}). Moreover it is well known that in dimension
one $\om$ is an irreducible fractional ideal and
$\lambda_R((\om:\m)/ \om)=1$. Thus $(\om:\m) \subset \bar R$
(otherwise, if $x \in \bar R \setminus  \om$, we have $\om =
(\om:\m) \cap (\om,x)$, a contradiction) and by \cite [Proposition
16]{b-f} $z=1$ is a minimal reduction of $(\om:\m)$. So $\om
\subseteq (\om:\m) \subseteq (\m:\m)$ and $R$ is almost
Gorenstein.

As for the last part of the statement, we have already proved that
in this case $z^{-1}(\om:\m)=(\m:\m)$;
then it is enough to apply the formula of Proposition \ref{type2}.
\end{proof}

%
%
%
%

\subsection{The general case}

Let $(R,\m)$ be a Cohen-Macaulay local ring of dimension $d$,
with canonical module $\omega_R$.
The goal of this
subsection is to prove that the property of being almost
Gorenstein for $\RR$ is independent of the choice of $a$ and $b$
also in the case $d>1$. We recall we are assuming that $R/\m$
is infinite.
The next lemma is proved e.g. in the proof of \cite[Proposition 3.3.18]{b-h},
but we include the short proof for the sake of completeness.

\begin{lem} \label{height}
Let $I$ be a regular ideal and a maximal CM $R$-module. Then $I$ has height one and
$R/I$ is a Cohen-Macaulay ring of dimension $d-1$.
\end{lem}

\begin{proof}
Since $I$ is a maximal CM $R$-module, $\depth I = \depth R$;
moreover $I$ has positive height, because is regular.
Hence the depth lemma \cite[Proposition 1.2.9]{b-h} implies that
\begin{gather*}
\dim R -1 \geq \dim R - \height I = \dim R/I \geq \depth R/I \geq \\
\geq \min \{\depth I -1,\depth R\} =\depth R -1 = \dim R -1.
\end{gather*}
Therefore all inequalities above are equalities and the thesis
follows immediately.
\end{proof}

The following lemma allows us to reduce
to the one-dimensional case.

\begin{lem}\label{lemma}
Let $x$ be an element of the ring $R$, that determines a
non-zerodivisor on $R/I$ (i.e. $(I:x)=I$); then
$$
\frac{\RR}{x\RR} \cong \frac{R}{xR}\left(
\frac{I+xR}{xR}\right)_{\overline{a},\overline{b}}
$$
where $\overline{a}$ and $\overline{b}$ are the images of $a$ and $b$ in $R/xR$.
\end{lem}

\begin{proof}
It is not difficult to check that there is a surjective ring
homomorphism $\alpha: \RR \rightarrow \frac{R}{xR}\left(
\frac{I+xR}{xR}\right)_{\overline{a},\overline{b}} $ defined by
$r+it \mapsto (r+xR) + (i+xR)t$. The assumption on $x$ implies
that $I \cap xR=xI$; hence $i \in xR$ if and only if $i=xj$ with
$j \in I$; therefore $\Ker (\alpha)= x\RR$\ and we obtain the thesis.
\end{proof}


In the next proposition we will use some results about filtrations
and superficial elements that can be found, e.g, in \cite[Chapter
1]{RV}.

\begin{prop}\label{AGd>1}
Let $d\geq1$ and let $I$ be a regular ideal of $R$. Then the almost Gorensteinness of $\RR$ does not depend on
the choice of $a$ and $b$.
\end{prop}

\begin{proof}
By Theorem \ref{conditions} it is enough to consider the case $d>1$.
Assume that there exist two elements $a', b' \in R$ for which
$R(I)_{a',b'}$ is almost Gorenstein. We have to show that $\RR$ is
almost Gorenstein for any $a,b \in R$. By Corollary \ref{gor2},
we can also assume that $R(I)_{a',b'}$ is not Gorenstein.

Our assumption means that there exists an exact sequence of
$R(I)_{a',b'}$-modules
$$
0 \rightarrow R(I)_{a',b'} \rightarrow \omega_{R(I)_{a',b'}} \rightarrow C \rightarrow 0,
$$
where the number of elements of a minimal system of generators of
$C$ equals its multiplicity. Let $M$ be the maximal ideal of
$R(I)_{a',b'}$ and consider the filtration of $C$ induced by $M$:
$$
C \supseteq MC \supseteq M^2C \supseteq \dots \supseteq M^i C
\supseteq \dots \ ;
$$
this is an $M$-filtration of the $R(I)_{a',b'}$-module $C$, but,
if we consider $C$ as an $R$-module, it is also an
$\m$-filtration. Therefore, we know that in $R$ there exists a
$C$-superficial sequence for $\m$ of length $d-1$; by definition
it is clear that it is also a $C$-superficial sequence for $M$.
Moreover, we can choose a sequence ${\bf f}=f_1, \dots, f_{d-1}$
that is also $R$-regular and, since $I$ has height one, such that
$I+({\bf f})$ is $\m$-primary (see \cite[Corollary 8.5.9]{HS}).
Consequently ${\bf f}$ is an $\RR$-regular
sequence for any $a,b \in R$
and the ideal of $R/I$ generated
by the classes $\overline f_1, \dots, \overline f_{d-1}$ is
$\m/I$-primary; therefore ${\bf \overline{f}}$ is a regular sequence,
because $R/I$ is a CM ring of dimension $d-1$ by Lemma \ref{height}.
Hence we can use the previous lemma and, from
\cite[Theorem 3.7 (2)]{GTT}, it follows that
$$
\frac{R(I)_{a',b'}}{ {\bf f}R(I)_{a',b'}} \cong
\frac{R}{{\bf f}R}\left( \frac{I+{\bf f}R}{{\bf f}R} \right)_{\overline{a'},\overline{b'}}
$$
is almost Gorenstein of dimension $1$; by Theorem \ref{conditions},
$$
\frac{R}{{\bf f}R}\left( \frac{I+{\bf f}R}{{\bf f}R}
\right)_{\overline{a},\overline{b}}
$$
is an almost Gorenstein ring for any $\overline{a}, \overline{b}
\in R/{\bf f}R$. Observe also that, as above, the ideal
$(I+{\bf f}R)/{\bf f}R$ is $\m/{\bf f}R$-primary.

Finally, since ${\bf f}$
is an $\RR$-regular sequence, this implies that $\RR$ is almost
Gorenstein for any $a,b \in R$, by \cite[Theorem 3.7 (1)]{GTT}.
\end{proof}

By Proposition \ref{AGd>1}, $\RR$ is almost Gorenstein if and only
if  $R(I)_{0,0} \cong R \ltimes I$ is almost Gorenstein. We have
already observed that if $I$ and $J$ are two isomorphic ideals,
then $\RR$ and $R(J)_{a,b}$ do not need to be isomorphic (cf.
Example \ref {ex1}). Anyway, it is easy to see that this happens
for idealization, i.e. if $N_1$ and $N_2$ are two isomorphic
$R$-modules then $R \ltimes N_1\cong R \ltimes N_2$. Thus applying
Proposition \ref{AGd>1} we obtain the following corollary.

\begin{cor}
If $I,J$ are two isomorphic  regular ideals of $R$, then $\RR$ is
almost Gorenstein if and only if $R(J)_{a,b}$ is almost
Gorenstein.
\end{cor}

In \cite{GTT} the authors study when the idealization
is almost Gorenstein. Proposition \ref{AGd>1} implies the
following generalization of \cite[Theorem 6.1]{GTT}.

\begin{cor} Let $I$ be a regular ideal of $R$ and assume that $I^{\vee}$ is isomorphic to a regular ideal
of $R$. Then
 the following are equivalent: \\
{\rm(1)} $\RR$ is almost Gorenstein for some $a,b \in R$; \\
{\rm (2)} $\RR$ is almost Gorenstein for all $a,b \in R$; \\
{\rm (3)}  $I$ is a maximal CM $R$-module and any proper
ideal $J$ of $R$ isomorphic to
$I^{\vee}$ is such that
 $f_1 \in J, \  \m(J+Q)=\m Q$, and $(J+Q)^2=Q(J+Q)$, for some parameter ideal $Q=(f_1, \dots, f_d)$ of $R$.
\end{cor}

\begin{proof}
The equivalence between {\rm (1)} and {\rm (2)}
follows from Proposition \ref{AGd>1}.

\noindent (2) $\Rightarrow$ (3): $I$ is a maximal CM $R$-module by
Proposition \ref{CM}, therefore $J \cong I^{\vee}$ is a maximal CM
$R$-module by \cite[Theorem 3.3.10]{b-h} and it is also isomorphic
to a regular ideal; moreover $J^{\vee} \cong I$, thus the
idealizations $R \ltimes I$ and $R \ltimes J^{\vee}$ are
isomorphic. Furthermore, by Lemma \ref{height}, it follows that
$R/J$ is a CM ring of dimension $d-1$. Thus it follows from
\cite[Theorem 6.1, (1) $\Rightarrow$ (2)]{GTT} that (3) holds
because
 $   R(I)_{0,0} \cong R \ltimes I \cong R \ltimes J^{\vee}$ is almost
Gorenstein.

\noindent (3) $\Rightarrow$ (1): By \cite[Theorem 6.1, (2) $\Rightarrow$ (1)]{GTT} we have that $R \ltimes J^{\vee} \cong R(I)_{0,0}$ is almost Gorenstein.
\end{proof}

\end{document}